\documentclass{amsart}
\usepackage[utf8]{inputenc}

\title{Double EPW-sextics with actions of $\mathcal{A}_7$ and irrational GM threefolds}

\author[S.\ Billi]{Simone Billi}
\address{University of Genova}
 \email{{\tt simone.billi@edu.unige.it}}
 
 \author[T.\ Wawak]{Tomasz Wawak}
\address{Jagiellonian University}
 \email{{\tt tomasz.wawak@doctoral.uj.edu.pl}}

 \date{\today}

\usepackage[english]{babel}
\usepackage[T1]{fontenc}
\usepackage{amssymb,amsmath,amsthm,amsfonts}
\usepackage{thmtools,mathtools}
\usepackage{microtype}
\allowdisplaybreaks
\usepackage{enumitem}
\usepackage{tikz,tikz-cd}
\usepackage[breaklinks,pdfencoding=auto,psdextra]{hyperref}
\usepackage{bm} 
\usepackage[scaled]{beramono}

\usepackage{listings}
\usepackage{caption}
\usepackage{xcolor}
\usepackage{amsmath}
\usepackage{amssymb}
\usepackage{amsfonts}
\usepackage{mathrsfs}
\usepackage{amsthm}
\usepackage{enumerate}
\usepackage{graphicx}
\usepackage{floatflt}
\usepackage{epstopdf}
\usepackage{comment}
\usepackage{tikz-cd}
\usepackage{url}
\usepackage{appendix}

\lstdefinelanguage{Macaulay2}{
comment=[l]{--},
alsoletter={'},
}
\newcommand{\hookuparrow}{\mathrel{\rotatebox[origin=c]{0}{$\hookrightarrow$}}}

\DeclareMathOperator{\Aut}{Aut} 
\DeclareMathOperator{\Homology}{H}
\DeclareMathOperator{\Gr}{Gr}
\DeclareMathOperator{\GL}{GL}
\DeclareMathOperator{\SL}{SL}
\DeclareMathOperator{\PGL}{PGL}
\DeclareMathOperator{\Jac}{Jac}
\DeclareMathOperator{\Alb}{Alb}
\DeclareMathOperator{\End}{End}
\DeclareMathOperator{\NS}{NS}
\DeclareMathOperator{\T}{T}
\DeclareMathOperator{\Spec}{Spec}

\DeclareFontShape{OT1}{cmtt}{bx}{n}{<5><6><7><8><9><10><10.95><12><14.4><17.28><20.74><24.88>cmttb10}{}
\newtheorem{thm}{Theorem}[section]
\newtheorem*{thm*}{Theorem}
\newtheorem{lem}[thm]{Lemma}

\newtheorem{prop}[thm]{Proposition}
\newtheorem{cor}[thm]{Corollary}
\newcommand{\hk}{hyper-K{\"a}hler}
\newcommand{\Hk}{Hyper-K{\"a}hler}
\newcommand{\KTsquare}{\hk{} fourfold of type K3$^{[2]}$}
\newcommand{\KTsquares}{\hk{} fourfolds of K3$^{[2]}$-type}
\begin{document}
\maketitle

\begin{abstract}
    We construct two examples of projective \KTsquares{} with an action of the alternating group $\mathcal{A}_7$. They are realized as double EPW-sextics and this allows us to construct two explicit families of irrational Gushel-Mukai threefolds.
\end{abstract}
\section*{Introduction}

We construct two explicit examples of \KTsquares{} with symplectic actions of the alternating group $\mathcal{A}_7$. According to the classification in \cite{Hohn2019FiniteGO}, this is one of the maximal finite groups that can act on a manifold of this type via symplectic automorphisms.  We describe the manifolds as double EPW-sextics. 
For each of them, the group of symplectic automorphisms $\mathcal{A}_7$ fixes their canonical ample divisor.

Consider a \(6\)-dimensional complex vector space \(V_6\).
There is a linear representation of \(\mathcal{A}_7\) on \( \bigwedge^3 V_6\) which decomposes into the direct sum of two Lagrangian subrepresentations \(A_1,A_2\); we set \(\mathbb{A}=A_1\) or \(A_2\). There is a canonical way to associate a double EPW-sextic to a general Lagrangian subspace of \( \bigwedge^3 V_6\), as explained in Section \ref{EPW_preliminary}. Our main result can be formulated as follows.
\begin{thm}\label{thm:maint}
     The double EPW-sextic associated to the Lagrangian subspace \(\mathbb{A}\) is a smooth polarized \hk{} fourfold \((Y_\mathbb{A},H_\mathbb{A})\). Moreover,  
    \[ \Aut_{H_\mathbb{A}}(Y_\mathbb{A})\cong\mathbb{Z}/2\mathbb{Z}\times \mathcal{A}_7 \]
    where \(\mathcal{A}_7\) corresponds to the subgroup of symplectic automorphisms. Furthermore, $(Y_{A_1}, H_{A_1})$ and $(Y_{A_2}, H_{A_2})$ are non-isomorphic as polarized manifolds.
\end{thm}
We can associate a family of GM varieties to such Lagrangian spaces as explained in Section \ref{GM_preliminary}. As an application, we show that for \(\mathbb{A}\) as above, we have the following result.
\begin{cor}\label{cor:mainc}
    Any smooth GM threefold associated to \(\mathbb{A}\) is irrational.
\end{cor}
The only other known explicit family of irrational GM threefold was constructed in \cite{debarre2021gushelmukai} and our result provides other two such families. 

This article is primarily inspired by \cite{debarre2021gushelmukai}, from where many ideas have been taken and adjusted to our case. In fact, the search for symmetric \hk{} manifolds has been of interest in recent years. One can find attempts at classification and explicit constructions in \cite{BH, BS} for K3 surfaces, in \cite{Hohn2019FiniteGO, Song, wawak, veryspecial} for \KTsquares{}. Our results should be seen as a continuation of these efforts.

The structure of the paper is the following: in the first section, we recall known facts about \hk{} manifolds and the lattice structure on their second integral cohomology group, EPW-sextics, and GM varieties. In the second section, we outline the construction of our examples and in the third, we prove the irrationality of the associated GM threefolds. In the appendix, we discuss the computation we performed with {\tt Macaulay2} \cite{M2}.

\section*{Acknowledgments}
We thank our advisors Giovanni Mongardi and Grzegorz Kapustka for their advice and supervision of the project, Micha\l\ Kapustka for helpful suggestions about the computations and thoughtful discussions, Jerzy Weyman for explanations on representation theory, and Lars Harvald Halle for useful comments concerning computations over a finite field.  During the writing of this paper, the second author was supported by the project Narodowe Centrum Nauki 2018/30/E/ST1/00530 .

\section{Preliminaries}\label{preliminaries}
This section is merely a collection of known facts about hyper-Kähler manifolds and associated lattices, double EPW-sextics, and GM varieties. 

\subsection{\Hk{} manifolds and their lattices}
A \hk{} manifold is a compact, simply connected complex K\"ahler manifold such that $\Homology^{2,0}(X) \cong \Homology^0(X, \bigwedge\nolimits^2\Omega_X) \cong \mathbb{C}\sigma_X$ where $\sigma_X$ is a nowhere degenerate holomorphic 2-form. For a \hk{} manifold, the group $\Homology^2(X, \mathbb{Z})$ has a canonical lattice structure, more precisely it is a free $\mathbb{Z}$-module with an integral bilinear form of signature \((3,b_2(X)-3)\). The N{\'e}ron-Severi group (or lattice) of $X$ is defined as
\begin{align*}
    \NS(X) = \Homology^{2}(X, \mathbb{Z}) \cap \Homology^{1,1}(X) \subset \Homology^2(X, \mathbb{C}).
\end{align*}
The transcendental lattice $\T(X)$ is its orthogonal complement with respect to the bilinear form. The following is an easy to prove corollary of \cite[Proposition 4.1]{survey}.
\begin{prop}\label{finite_group}
    Let $X$ be a projective \hk{} manifold. A group of automorphisms $G\subset \Aut(X)$ is finite if and only if it fixes an ample class on $X$.
\end{prop}
For the general theory of \hk{} manifolds, we refer for example to \cite{survey}.

\subsection{EPW-sextics, their double covers and automorphisms}\label{EPW_preliminary}
EPW-sextics are singular sextic hypersurfaces in $\mathbb{P}^5$, first constructed by Eisenbud, Popescu and Walter (\cite{eisenbud2001lagrangian}).  O'Grady showed in \cite{o2006irreducible} that they have a natural double cover, which is in general a smooth \KTsquare{}. We briefly recall the construction.

Fix a \(6\)-dimensional $\mathbb{C}$-vector space $V_6$ and choose an isomorphism $\bigwedge\nolimits^6 V_6\cong \mathbb{C}$. It induces a symplectic form $(\alpha,\beta):=\alpha\wedge\beta$ on $\bigwedge\nolimits^3 V_6\cong \mathbb{C}^{20}$. Denote by \(\mathbb{LG}(\bigwedge\nolimits^3V_6)\) the Grassmannian of Lagrangian subspaces. 
Consider the vector subbundle $F\subset \bigwedge\nolimits^3 V_6\otimes \mathcal{O}_{\mathbb{P}(V_6)}$ whose fiber over $[v]\in\mathbb{P}(V_6)$ is given by:
$$ F_v=\{\alpha\in\bigwedge\nolimits^3 V_6 \mid \, \alpha \wedge v=0\}.$$
To an element \( A\in\mathbb{LG}(\bigwedge\nolimits^3V_6)\) one associates the loci
\begin{align*}
    Y_A[k] &= \{ [v]\in \mathbb{P}(V_6)\mid\dim(A\cap F_v)\geq k\},
\end{align*} and
$Y_A = Y_A[1]$ is called an \textit{EPW-sextic}. We define two divisors:
\begin{align*}
    \Sigma &=\{A \in \mathbb{LG}(\bigwedge\nolimits^3V_6) \, \mid \, \mathbb{P}
    (A)\cap \Gr(3,V_6)=\emptyset \} \\
    \Delta&=\{ A \in \mathbb{LG}(\bigwedge\nolimits^3V_6) \, \mid \,Y_A[3] = \emptyset \}.
\end{align*} 

For \(A\in\mathbb{LG}(\bigwedge\nolimits^3V_6)\smallsetminus(\Sigma\cup\Delta)\) there is a canonical double cover $\pi_A\colon\widetilde{Y}_A\rightarrow Y_A$ branched along the surface $Y_A[2]$, where \(\widetilde{Y}_A\) is a smooth \KTsquare{} called \textit{double EPW-sextic}, originally defined in \cite{o2006irreducible}. Moreover, $\widetilde{Y}_A$ carries a canonical polarization $H=\pi_A^*\mathcal{O}_{Y_A}(1)$ and the image of the morphism $\widetilde{Y}_A\rightarrow \mathbb{P}(\Homology^0(\widetilde{Y}_A,H)^{\vee})$ is isomorphic to $Y_A$.

Recall that if \(A\not\in\Sigma\) the automorphisms of the sextics are all linear:
\begin{equation}\label{eq1}
    \Aut(Y_A) =\{g\in \PGL(V_6)\mid(\bigwedge\nolimits^3g)(A)=A\}
\end{equation}
and this is a finite group by \cite[Proposition B.9]{2018}.

 Consider the embedding $\Aut(Y_A)\hookuparrow \PGL(V_6)$ and let $G$ be the inverse image of $\Aut(Y_A)$ via the canonical map $\SL(V_6)\rightarrow \PGL(V_6)$. Then, $G$ is an extension of $\Aut(Y_A)$ by the cyclic group $\langle \gamma\rangle$ of order $6$, so we have an induced representation of $G$ on $\bigwedge\nolimits^3 V_6$ and this factors through a representation of $\widetilde{\Aut}(Y_A):=G/\langle \gamma^2\rangle$. Since $A$ is preserved by this action, we have a morphism of central extensions 
 \begin{equation}\label{diagram}
    \begin{tikzcd}
1 \arrow[r] & \langle \gamma^3\rangle \arrow[r] \arrow[d, hook] & \widetilde{\Aut}(Y_A) \arrow[r] \arrow[d] & \Aut(Y_A) \arrow[r] \arrow[d] & 1 \\
1 \arrow[r] & \mathbb{C}^* \arrow[r]                            & \GL(A) \arrow[r]                      & \PGL(A) \arrow[r]    & 1
\end{tikzcd}
 \end{equation}
 and by \cite[Lemma A.1]{debarre2021gushelmukai} the vertical maps are injective.
 
 Denote by $\Aut_H(\widetilde{Y}_A)$ the group of automorphisms fixing the polarization \(H\). 
 \begin{prop}[Kuznetsov, {\cite[Proposition A.2]{debarre2021gushelmukai}}]\label{kuz}
 Let $A\notin \Sigma$ be a Lagrangian space. There is an isomorphism 
 \[\Aut_H(\widetilde{Y}_A)\cong \Aut(Y_A)\times \langle\iota\rangle\] where \(\iota\) is the anti-symplectic involution associated to the double covering \(\pi_A\) and the factor $\Aut(Y_A)$ corresponds to the subgroup of symplectic automorphisms.
 \end{prop}
 When \(A\notin \Sigma\), there is a canonical connected double covering \cite[Theorem 5.2(2)]{debarre2019double}
\begin{equation*}
    \widetilde{Y}_A[2]\rightarrow Y_A[2]
\end{equation*} and by \cite[Proposition A.6 (Kuznetsov)]{debarre2021gushelmukai} we have an injection $\widetilde{\Aut}(Y_A)\hookuparrow \Aut(\widetilde{Y}_A[2])$.

Recall that the analytic representation of a finite group $G$ acting on an Abelian variety $X$ is the composition 
\begin{equation}
    G\rightarrow \End_\mathbb{Q}(X)\rightarrow\End_\mathbb{C}(T_{X,0}).
\end{equation}

We recall the following useful result, where \(\Alb(\widetilde{Y}_A[2])\) denotes the Albanese variety of \(\widetilde{Y}_A[2]\).
\begin{prop}[{\cite[Proposition A.7]{debarre2021gushelmukai}}]\label{analiticrep}
Suppose that $Y_A[3]=\emptyset$. The restriction to the subgroup $\widetilde{\Aut}(Y_A)$ of the analytic representation of $\Aut(\widetilde{Y}_A[2])$ on $\Alb(\widetilde{Y}_A[2])$ is the injective middle vertical map in the diagram (\ref{diagram}).
\end{prop}

\subsection{GM varieties and their automorphisms}\label{GM_preliminary}
Let $V_5$ be a $5$-dimensional complex vector space.
A GM variety (shorthand for ordinary Gushel-Mukai) of dimension $n\in\{3,4,5\}$ is the smooth complete intersection of the Grassmannian $\Gr(2,V_5)\subset \mathbb{P}(\bigwedge\nolimits^2 V_5)$ with a linear space $\mathbb{P}^{n+4}$ and a quadric. These varieties are Fano varieties with Picard number $1$, index $n-2$ and degree $10$.

There is a bijection between the set of isomorphism classes of GM varieties of dimension $n$ and isomorphism classes of triples $(V_6,V_5,A)$, where $A\in \mathbb{LG}(\bigwedge\nolimits^3V_6)\smallsetminus\Sigma$ and $V_5\subset V_6$ is a hyperplane that satisfies 
\begin{equation*}
    \dim(A\cap \bigwedge\nolimits^3V_5)=5-n.
\end{equation*}
The correspondence is outlined in \cite[Theorem 3.10 and Proposition 3.13(c)]{2018}.

In particular, for \(A\in \mathbb{LG}(\bigwedge\nolimits^3V_6)\smallsetminus(\Sigma\cup\Delta)\) there is an associated double EPW-sextic and associated families of GM varieties of dimensions \(3,4\) and \(5\).

\section{EPW-sextics with an action of $\mathcal{A}_7$}
 The general idea is to find a Lagrangian \(A\subset \bigwedge^3 V_6\) which is invariant under the action of the group $\mathcal{A}_7$, to get an invariant EPW-sextic. According to \cite{Wilson1985ATLASOF}, there exists a group denoted by $3.\mathcal{A}_7$ with $\eta$ an element of order $3$ such that $3.\mathcal{A}_7/\langle \eta\rangle\cong \mathcal{A}_7$. This group has a unique irreducible representation $\rho\colon 3.\mathcal{A}_7\rightarrow \mathbb{C}^6\cong V_6$ and \(\eta\) acts by multiplication for a third root of unity. Generators for this representation can be found in \cite{ATLAS}. Notice that this induces a representation on $\bigwedge\nolimits^3 V_6$ which descends to a representation of the quotient $\mathcal{A}_7$ on $\bigwedge\nolimits^3 V_6$. Denote this (faithful) representation of $\mathcal{A}_7$ by $W$.

The irreducible complex representations of $\mathcal{A}_7$ of dimension smaller than or equal to $20$ have dimensions $1,6,10,10,14,14$, and $15$ and are described in \autoref{Characters}.
\begin{table}[ht]
\scalebox{.8}{
\def\arraystretch{1.5}
$\begin{array}{|r|ccccccccc|}
\hline
\text{Conj.\ class}&id&[ab^{-1}ab]&[a]&[a^{-1}bab]&[a^{-1}bab^2]&[b]&[ababab^2]&[ab]&[a^{-1}b]\\ 
\hline W_0& 1   & 1  & 1  & 1  & 1   & 1   & 1               & 1  & 1             \\
W_6    & 6  &2&3&0&0&1&-1&-1&-1\\
W_{10} & 10  &-2&1&1&0&0&1&-\frac{1}{2}(1-i\sqrt{7}) & -\frac{1}{2}(1+i\sqrt{7})\\
W'_{10}  & 10  &-2&1&1&0&0&1&-\frac{1}{2}(1+i\sqrt{7}) & -\frac{1}{2}(1-i\sqrt{7})\\
W_{14} & 14  &2&2&-1&0&-1&2&0 & 0 \\
W'_{14}  & 14  &2&-1&2&0&-1&-1&0 & 0\\
W_{15}  & 15  &-1&3&0&-1&0&-1&1 & 1\\   
\hline
W& 20  &-4&2&2&0&0&2& -1& -1   \\
\hline
\end{array}$}
\caption{\label{Characters}Character table of $\mathcal{A}_7$, according to \cite{Wilson1985ATLASOF}.}
\end{table}
\begin{lem}
The representation $W$ decomposes as the direct sum \(W=A_1\oplus A_2 \) of the only two irreducible 10-dimensional representations $R_1=(A_1,\rho_1)\cong W_{10}$ and $R_2=(A_2,\rho_2)\cong W'_{10}$ of the group $\mathcal{A}_7$. Moreover the underlying vector spaces $A_1,A_2\subset\bigwedge\nolimits^3 V_6$ of those representations are Lagrangian.
\end{lem}
\begin{proof}
The fact that $W$ has the mentioned decomposition is just a computation of characters (see \autoref{Characters}). The fact that the subrepresentations are Lagrangian is easily checked with computer algebra after determining the invariant spaces with a projection formula. Auxiliary files with {\tt GAP} \cite{GAP4} codes are provided as attachments. 
\end{proof}

As a consequence of (\ref{eq1}), setting $\mathbb{A}=A_1$ or $A_2$ leads to an EPW-sextic $Y_{\mathbb{A}} \subset \mathbb{P}^5$ which is invariant under the action of $\mathcal{A}_7$. The two sextics $Y_{A_1}$ and $Y_{A_2}$ are dual (cf. \cite[Section 3]{o2006irreducible}). From now on, $\mathbb{A}$ will denote one of the Lagrangian spaces \(A_1\) or \(A_2\).
\begin{prop}\label{nodec}
    We have $\mathbb{A} \in \mathbb{LG}(\bigwedge\nolimits^3V_6)\smallsetminus(\Sigma\cup\Delta)$, hence the double cover $\widetilde{Y}_\mathbb{A}\rightarrow Y_\mathbb{A}$ is a smooth \hk{} fourfold.
\end{prop}

\begin{proof}
Following \cite{o2012epw}, we need to prove that $\mathbb{A}$ does not belong to $\Sigma$ and $\Delta$.

The computations with {\tt Macaulay2} \cite{M2} in Appendix \ref{deg40locus_code} show that $Y_\mathbb{A}[3]$ is empty, hence \(\mathbb{A}\notin \Delta\). 

According to \cite[Proposition 1.5]{ferretti2009chow}, the singular locus of \(Y_A\) is given by the union of the 40-degree surface $Y_\mathbb{A}[2]$ and the planes of the form $\mathbb{P}(U)$, where $U$ is a three-dimensional subspace of $V_6$ such that $\bigwedge\nolimits^3 U\subset \mathbb{A}$. 
The computations in Appendix \ref{deg40locus_code} show that the singular locus of \( Y_\mathbb{A}\) has degree $40$ so it must coincide with $Y_\mathbb{A}[2]$ (cf. \cite[Corollary 1.10]{o2012epw}), thus \(\mathbb{A}\notin\Sigma\). This completes the proof.

\end{proof}

Denote by $\Aut^s_H(\widetilde{Y}_A)$ the group of automorphisms of \(\widetilde{Y}_A\) that are symplectic and fix the polarization \(H\).

\begin{prop} \label{isom}There is an isomorphism $\Aut^s_H(\widetilde{Y}_\mathbb{A})\cong \mathcal{A}_7$. 
\end{prop}
\begin{proof} Since \(\mathbb{A}\not\in\Sigma\), we know by \autoref{kuz} that \(\Aut(Y_\mathbb{A})\cong\Aut^s_H(\widetilde{Y}_\mathbb{A})\). By construction, the action of \(\mathcal{A}_7\) on \(\bigwedge^3 V_6\) is induced by linear transformations of \(V_6\). Using equality (\ref{eq1}) combined with the previous isomorphism we get the inclusion \(\mathcal{A}_7\subseteq\Aut^s_H(\widetilde{Y}_\mathbb{A}) \). From \autoref{finite_group}, it follows that $\Aut^s_H(\widetilde{Y}_\mathbb{A})$ is finite,  and by the maximality of the group in \cite[Theorem A and Table 6]{Hohn2019FiniteGO}, one concludes that the inclusion $\mathcal{A}_7\subseteq \Aut^s_H(\widetilde{Y}_\mathbb{A})$ is in fact an equality.
\end{proof}

Now we are ready to show that the two examples we found, $\widetilde{Y}_{A_1}$ and $\widetilde{Y}_{A_2}$, are not isomorphic as polarized manifolds. We will need the following lemmas.

\begin{lem}\label{proj_repr}The projective representations of $\mathcal{A}_7$ induced by \(R_1=(A_1,\rho_1)\) and \(R_2=(A_2,\rho_2)\) are not isomorphic as projective representations.\end{lem}\begin{proof}Suppose by contradiction that the two representations are projectively isomorphic. By assumption there exist a linear map \(\phi \colon A_1 \rightarrow A_2\) and a homomorphism \(\mu \colon \mathcal{A}_7 \rightarrow \mathbb{C}^*\) such that for any \(g\in\mathcal{A}_7\)\[\rho_1(g)  = \mu(g) \cdot (\phi^{-1} \circ \rho_2(g) \circ \phi) ,\]which implies \(\chi_{\rho_1}(g)=\mu(g) \chi_{\rho_2}(g)\).
Consider generators \(a, b \in \mathcal{A}_7\) as in \autoref{Characters}, observe that \(\chi_{\rho_1}(a) = \chi_{\rho_2}(a)\) and hence \(\mu(a)=1\). By multiplicativity, we have \(\mu(a^{-1}b)=\mu(a b)\), on the other hand \(\mu(a^{-1}b) = \frac{x}{y}\) and \(\mu(ab) = \frac{y}{x}\), where \(x = -\frac{1}{2}(1+i\sqrt{7})\) and \(y=-\frac{1}{2}(1-i\sqrt{7})\). This leads to a contradiction, since \(\frac{x}{y} \not= \frac{y}{x}\).\end{proof}

\begin{lem}\label{lem_not_isomorphic}
There is no $f \in \GL(V_6)$ such that $(\bigwedge^3 f)(A_1)= A_2$.
\end{lem}
\begin{proof}Assume for a contradiction that there is such an \(f\in \GL(V_6)\). 
Consider the representations $R_1=(A_1,\rho_1)$ and $R_2=(A_2,\rho_2)$ and set \(h:=(\bigwedge^3 f)_{\mid A_1}\colon A_1\rightarrow A_2.\)
Notice that there is an isomorphism of representations
\begin{align*}
      (A_2,\widetilde{\rho})\cong R_1,
\end{align*}
where \(\widetilde{\rho}(g)=h \circ \rho_1(g) \circ h^{-1}\) for \(g\in\mathcal{A}_7\), and recall that \(R_1\) is not isomorphic to \(R_2\). Let \(G\subset\PGL(V_6)\) be the subgroup generated by the classes of linear maps of the form \(\rho_2(g)\) and \(h\circ \rho_1(g)\circ h^{-1}\) for \(  g\in \mathcal{A}_7\). 
There are inclusions
\[\rho_2(\mathcal{A}_7) \subset G\subset \Aut(Y_{A_2}),\]
where the second inclusion follows from (\ref{eq1}) since all the elements of \(G\) are expressed by third wedges of linear transformations of $V_6$ which preserve the Lagrangian space $A_2$, and the first inclusion is strict. In fact, if equality held that would mean that the projective representations induced by \(R_1\) and \(R_2\) were equivalent, which is not the case by virtue of \autoref{proj_repr}. We conclude using the isomorphisms 
\begin{align*}
    \Aut(Y_{A_2})\cong \Aut_H^s(\widetilde{Y}_{A_2})\cong\mathcal{A}_7
\end{align*} from Propositions \ref{kuz} and \ref{isom} to get a contradiction.

\end{proof}

\begin{prop}\label{non_iso}
    The manifolds $(\widetilde{Y}_{A_1}, H_1)$ and $(\widetilde{Y}_{A_2}, H_2)$ are not isomorphic as polarized manifolds where $H_i = \pi_{A_i}^*\mathcal{O}_{Y_{A_i}}(1)$ for $i\in\{ 1,2\}.$
\end{prop}
\begin{proof}
    By \cite[page 486]{o2015period}, if $A_1, \, A_2 \in \mathbb{LG}(\bigwedge\nolimits^3V_6)\smallsetminus(\Sigma\cup\Delta)$ are not in the same orbit of $\PGL(\bigwedge^3 V_6)$, then $\widetilde{Y}_{A_1}$ and $\widetilde{Y}_{A_2}$ have different periods, so they cannot be isomorphic. Lemma \ref{lem_not_isomorphic} finishes the proof.
\end{proof}

We can note that as $Y_{A_1}$ and $Y_{A_2}$ are dual to each other, that makes $\widetilde{Y}_{A_1}$ and $\widetilde{Y}_{A_2}$ dual in the sense of the involution of the moduli space of double EPW-sextics described in \cite[Section 6]{o2006irreducible}. 

 A concatenation of the previous results gives a proof of the main theorem:
\begin{proof}[Proof of \autoref{thm:maint}]
    By \autoref{nodec}, we are in the hypothesis of \autoref{kuz} which describes the automorphism group of \(\widetilde{Y}_\mathbb{A}\) fixing the polarization. \autoref{isom} computes the group of symplectic automorphisms fixing the polarization. Finally, by \autoref{non_iso}, the examples are not isomorphic as polarized manifolds.
\end{proof}
We also obtain the following information on the constructed manifolds.
\begin{prop}
The transcendental lattice \(\T(\widetilde{Y}_{\mathbb{A}})\) is isomorphic to \begin{align*}
     \begin{pmatrix}
6 & 0 \\
0 & 70
\end{pmatrix},
\end{align*}
that is, there exists a basis $\{v_1, v_2\}$ of \ $\T(\widetilde{Y}_{\mathbb{A}})$ such that $v_1^2 = 6$, $v_2^2 = 70$, and $v_1 \cdot v_2 = 0$.
\end{prop}
\begin{proof}
By \cite{o2006irreducible}, the polarization $H$ discussed earlier satisfies $H^2 = 2$.
According to \cite[Table 1]{wawak}, a projective \KTsquare{} admitting an action of a nontrivial overgroup of $\mathcal{A}_7$ and fixing an ample vector $h$ such that $h^2 = 2$ in $\NS(X) \subset \Homology^2(X, \mathbb{Z})$ must have the transcendental lattice from the statement.
\end{proof}

    In \cite{Hohn2019FiniteGO}, the authors constructed an example of projective hyper-Kähler manifold of type K3\(^{[2]}\) with a symplectic action of \(\mathcal{A}_7\) as the Fano variety of lines of a cubic fourfold with an action of \(\mathcal{A}_7\). However, the manifold \(\widetilde{Y}_{\mathbb{A}}\) is not even birational to the Fano variety of lines of a cubic fourfold with an action of \(\mathcal{A}_7\), since the two manifolds have different transcendental lattices according to \cite[Table 1]{wawak}.

\section{Irrational GM threefolds}
Let $X_\mathbb{A}$ be any (smooth) GM variety of dimension $3$ or $5$ associated with $\mathbb{A}$ and let $\Jac(X_\mathbb{A})$ be its intermediate Jacobian, which is a \(10\)-dimensional principally polarized Abelian variety.

From the fact that $Y_\mathbb{A}[2]$ is smooth, by \cite[Theorem 1.1]{kuznetsov2020gushel} there are a canonical principal polarization $\theta$ on the Albanese variety $\Alb(\widetilde{Y}_\mathbb{A}[2])$ and a canonical isomorphism
\begin{equation}\label{jaco}
    (\Jac(X_\mathbb{A}),\theta_{X_\mathbb{A}})\cong(\Alb(\widetilde{Y}_\mathbb{A}[2]),\theta)
\end{equation}
between principally polarized Abelian varieties. Furthermore, the tangent spaces at the origin of these varieties are isomorphic to $\mathbb{A}$. 

\begin{prop}
The principally polarized Abelian variety $(\Jac(X_\mathbb{A}),\theta_{X_\mathbb{A}})$ is indecomposable.
\end{prop}
\begin{proof}
Suppose that this principally polarized Abelian variety is isomorphic to a product of $m\geq 2$ nonzero indecomposable principally polarized Abelian varieties.

Since $\mathcal{A}_7\cong\Aut(Y_\mathbb{A})$, the diagram (\ref{diagram}) reads:
 \begin{equation}
 \begin{tikzcd}
1 \arrow[r] & \langle \gamma^3\rangle \arrow[r] \arrow[d] & \widetilde{\mathcal{A}}_7 \arrow[r, "\psi"] \arrow[d, "\rho_a"'] & \mathcal{A}_7 \arrow[r] \arrow[d, dashed] \arrow[ld, "\rho"', dashed] & 1 \\
1 \arrow[r] & \mathbb{C}^* \arrow[r]                      & \GL(\mathbb{A}) \arrow[r, "\pi"]                                & \PGL(\mathbb{A}) \arrow[r]                                            & 1
\end{tikzcd}
 \end{equation}
where $\widetilde{\mathcal{A}}_7$ is an extension of $\mathcal{A}_7$ by the group of order two, $\rho_a$ is the analytic representation $\widetilde{\mathcal{A}}_7\rightarrow\GL(T_{\Jac(X_\mathbb{A}),0})$ by \autoref{analiticrep}, and $\rho$ is the irreducible representation $\mathbb{A}$. Now $ \rho_a\neq\rho \circ \psi$ since both representations are faithful, but the equality $\pi\circ \rho_a=\pi\circ\rho\circ \psi$ holds by construction (using the commutativity of the diagram (\ref{diagram})). This means that the two actions on $\mathbb{A}$ differ by a scalar multiplication, hence if one of the representations decomposes then the other must do so as well. We supposed that the Jacobian is a nontrivial product and so the analytic representation decomposes in the sum of the tangent spaces of the components, but this is a contradiction since $\mathbb{A}$ is irreducible as an $\mathcal{A}_7$-representation. 

\end{proof}
Now we are ready to prove \autoref{cor:mainc}.
\begin{proof}[Proof of \autoref{cor:mainc}]
The idea of proof is taken from \cite[Theorem 5.2]{debarre2021gushelmukai}.
We want to use the Clemens-Griffiths criterion: if the intermediate Jacobian of a Fano threefold is not the (polarized) product of Jacobians of curves, then the variety is irrational.

Since $(\Jac(X_\mathbb{A}),\theta_{X_\mathbb{A}})$ is indecomposable, we can reduce to the case where $(\Jac(X_\mathbb{A}),\theta_{X_\mathbb{A}})\cong(\Jac(C),\theta_C) $ for $C$ a curve of genus $10$. We have a faithful action of $\mathcal{A}_7$ on the Jacobian, by the Torelli theorem, the group of automorphisms of that principally polarized Abelian variety is either $\Aut(C)$ or $\Aut(C)\times \mathbb{Z}/2\mathbb{Z}$. In conclusion, $\mathcal{A}_7$ embeds in one of those two groups, but this is a contradiction since $|\mathcal{A}_7|=2520$ and $|\Aut(C)|<756$ by \cite[Theorem 3.7]{miranda1995algebraic}.

\end{proof}

\appendix 

\section{Macaulay computations}\label{codes}\label{deg40locus_code}

The following {\tt Macaulay2} code computes equations over a finite field for the EPW-sextic associated to a Lagrangian space, checks that the sextic is irreducible and has the right degree. It also computes the singular locus of the EPW-sextic and checks it is a surface of degree 40.
The two Lagrangian subrepresentations \(R_1\) and \(R_2\) are defined over $\mathbb{Q}[\xi_{21}]$ with \(\xi_{21}=\xi_3 \xi_7\) where \(\xi_3\) and \(\xi_7\) are respectively a third and a seventh root of unity (see \autoref{Characters}). In fact, \(\mathbb{Q}[i\sqrt{7}]\subset \mathbb{Q}[\xi_7]\subset\mathbb{Q}[\xi_{21}]\) since \(\alpha= \xi_7+\xi^2_7+\xi_7^4\) satisfies the quadratic equation \(\alpha^2+\alpha+2=0\) whose solutions are given by \(\frac{-1\pm i\sqrt{7}}{2}\). We run the computations over the finite field $\mathbb{F}_{127}$; the choice of \(p=127\) is such that the ideal \( (127)\) completely decomposes in the ring \(\mathbb{Z}[\xi_{21}]\) and hence the residue field at any prime in the decomposition is exactly \(\mathbb{F}_{127}\). Let \(\mathfrak{q}\) be a prime ideal in the primary decomposition of \((127)\subset \mathbb{Q}[\xi_{21}]\) and consider the discrete valuation ring \(D=\mathbb{Z}[\xi_{21}]_\mathfrak{q}\) given by localization at \(\mathfrak{q}\). We consider the proper map \(\mathbb{P}^5_D\rightarrow\Spec (D)\). The restriction to \(X[k]=Y_{\mathbb{A}}[k]\times_{\Spec(\mathbb{Z})}\Spec(D)\) is again proper for \(k=1,2,3\). The fiber \(X[k]_{(0)}\) over the the ideal \((0)\) corresponds to \(Y_{\mathbb{A}}[k]\), while the fiber \(X[k]_\mathfrak{q}\) over the ideal \(\mathfrak{q}\) corresponds to the reduction of \(Y_{\mathbb{A}}[k]\) over the residue field \(\mathbb{F}_{127}\) at \(\mathfrak{q}\). Observe that if \(X[k]_{\mathfrak{q}}\) is empty then \(X[k]_{(0)}\) is also empty. Moreover, \(X[k]\) is reduced and all its irreducible components dominate the Dedekind scheme \(\Spec(D)\), hence the family \(X[k]\rightarrow \Spec(D)\) is flat.

\begin{lstlisting}[language=Macaulay2]
p = 127;
F = ZZ/p;
R = F[v];
I = ideal (v^12-v^11+v^9-v^8+v^6-v^4+v^3-v+1); --21st cyclotomic polynomial 
J = decompose(I);
length J --returns 12
K = toField(R/J_0);
P = K[x,y,z,t,u,w];
--Matrix of coordinates
M = matrix {
    { 0,0,0,0,0,0,0,u,-t,z},
    {0,0,0,0,0,-u,t,0,0,-y},
    { 0,0,0,0,u,0,-z,0,y,0},
    {0,0,0,0,-t,z,0,-y,0,0},
    { 0,0,0,0,0,0,0,0,0,x},
    { 0,0,0,0,0,0,0,0,-x,0},
    { 0,0,0,0,0,0,0,x,0,0},
    { 0,0,0,0,0,0,0,0,0,0},
    { 0,0,0,0,0,0,0,0,0,0},
    { 0,0,0,0,0,0,0,0,0,0}
};
MM = M + transpose ( M );
--Mat1 is the symmetric matrix associated to the basis of the lagrangian A1
Lambda1 = Mat1-MM; --The EPW is given by 
d1 = det(Lambda1); --polynomial of degree 6, the equation for the EPW
I1 = ideal d1;
degree(I1) --returns 6
s1 = ideal singularLocus I1; --singular locus of the EPW over the finite field
dim(s1) --returns 3, so the projective dim is 2
degree(s1) --returns 40
\end{lstlisting}
\begin{lstlisting}[language=Macaulay2]
--The ideal generated by the 8x8 minors describes the locus Y_A1[3]
J1 = minors(8,Lambda1);
-- Dimesion of Y_A1[3] (affine chart)
dim(J1)
--returns -1
\end{lstlisting}
The code for computing the equations for EPW-sextics is based on the appendix of \cite{https://doi.org/10.48550/arxiv.2202.00301}. Here we give a sketch for the computation in an affine chart but the emptiness of a variety has to be checked on all the affine charts.

The locus $Y_\mathbb{A}[3]$ is computed as the zero locus of the ideal generated by the $8\times 8$ minors of the symmetric matrix that determines the EPW-sextic. The outcome of the computation is that the locus $Y_\mathbb{A}[3]$ is empty over the field \(\mathbb{F}_{127}\), and hence empty over \(\mathbb{Q}[\xi_{21}]\), proving that $\mathbb{A}\not\in \Delta$. The singular locus of \(X[1]_\mathfrak{q}\) is a surface of degree \(40\) and, by upper semicontinuity, the singular locus of \(X[1]_{(0)}=Y_{\mathbb{A}}\) has same degree and dimension. Hence, the singular locus of \(Y_{\mathbb{A}}\) must coincide with the smooth surface $Y_\mathbb{A}[2]$, since they both have degree $40$ \cite[Corollary 1.10]{o2012epw} and one is contained in the other. This is equivalent to $\mathbb{A}\not\in \Sigma$, under the condition \(\mathbb{A}\notin \Delta\).

\bibliographystyle{halpha-abbrv}
\bibliography{references}
\end{document}